\documentclass[reqno,12pt]{amsart}

\usepackage{enumerate}
\usepackage[margin=1in]{geometry}
\usepackage{ifpdf}
\usepackage{amsmath}
\usepackage{amsfonts}
\usepackage{amssymb}
\usepackage{amsthm}
\usepackage[hyperfootnotes=false]{hyperref}
\usepackage{setspace}
\usepackage{graphicx}
\usepackage{amsrefs}
\usepackage{nicefrac}

\newcommand{\ignore}[1]{}

\newcommand{\snorm}[1]{\lVert {#1} \rVert}

\newcommand{\R}{{\mathbb{R}}}

\newcommand{\bP}{{\mathbb{P}}}

\newcommand{\sC}{{\mathcal{C}}}

\newcommand{\sH}{{\mathcal{H}}}

\newtheorem{thm}{Theorem}[section]

\newtheorem{prop}[thm]{Proposition}

\newtheorem{cor}[thm]{Corollary}

\newtheorem{lemma}[thm]{Lemma}

\theoremstyle{definition}
\newtheorem{defn}[thm]{Definition}
\newtheorem{example}[thm]{Example}

\theoremstyle{remark}

\author{Orest Bucicovschi}
\thanks{The first author was supported in part by DARPA QuEST grant
N66001-09-1-2025.}
\address{Department of Mathematics, University of California at
San Diego, La Jolla, CA 92093-0112, USA}
\email{obucicov@math.ucsd.edu}

\author{Ji\v{r}\'i Lebl}
\thanks{The second author was supported in part by NSF grant DMS 0900885.}
\address{Department of Mathematics, University of Wisconsin-Madison, 
Madison, WI 53706, USA}
\curraddr{Department of Mathematics, Oklahoma State University,
Stillwater, OK 74078, USA}
\email{lebl@math.okstate.edu}

\date{September 12, 2013}

\ifpdf
\hypersetup{
  pdftitle={On the continuity and regularity of convex extensions},
  pdfauthor={Orest Bucicovschi and Jiri Lebl}
}
\fi

\title{On the continuity and regularity of convex extensions}

\begin{document}


\begin{abstract}
We study continuity and regularity of convex extensions of functions from a
compact set $C$ to its convex hull $K = \operatorname{co}(C)$.
We show that if $C$
contains the relative boundary of $K$, and $f$ is a
continuous convex function on $C$, then $f$ extends to
a continuous convex function on $K$ using the standard convex
roof construction.  In fact, a necessary and sufficient
condition for $f$ to extend from any set to a continuous
convex function on the convex hull is that $f$ extends to a continuous
convex function on the relative boundary of the convex hull.
We give examples showing that the hypotheses in the results
are necessary.  In particular, if $C$ does not contain the entire
relative boundary of $K$, then there may not exist any continuous
convex extension of $f$.
Finally, when the boundary of $K$ and $f$
are $C^1$ we give a necessary and sufficient condition 
for the convex roof construction to be $C^1$ on all of $K$.
We also discuss an application of the convex roof construction
in quantum computation.
\end{abstract}

\maketitle



\section{Introduction} \label{section:intro}

Extending convex functions is a problem with applications from economics
\cite{peterswakker} to quantum computing \cites{uhlmann1,uhlmann2,Wootters}
(see also \S~\ref{section:apptoquant}) and in a wide
variety of other optimization problems.
A classical theorem of Gale, Klee, and Rockafellar \cite{GKR} is that
given a convex domain, then every bounded
convex function defined on the interior extends continuously
to the boundary if and only if the domain is a polyhedron.  In this paper
we look at the opposite problem of extending a function from the boundary
to the interior.  For the set of extreme points this was also studied by
Lima~\cite{Lima}.

Given a convex function defined on a set, there exists a 
common construction called the \emph{convex roof} that defines
the largest convex extension to the convex hull of the original set.
In this paper we study 
the continuity and regularity properties of this construction.
Background in convex analysis is taken from the book
\cite{Rockafellar:book}.

For a convex set $K \subset \R^d$, let
$\operatorname{ri}(K)$ denote the \emph{relative
interior} of $K$.  That is, the topological interior
of $K$ inside the affine hull of $K$.
Then let $\operatorname{bd}(K) =  \overline{K}
\setminus \operatorname{ri}(K)$
denote the \emph{relative boundary}.
If $f \colon C \to \R$ is a function, where $C \subset \R^d$ is not
necessarily convex,
we say that $f$ is \emph{convex} if
\begin{equation}
f(x) \leq \sum_{j=1}^k t_j f(x_j) ,
\end{equation}
for any finite collection $x, x_1,\ldots,x_k \in C$ and
$t_1,\ldots,t_k \in [0,1]$, such that $\sum_j t_j = 1$ and
$\sum_j t_j x_j = x$.

\begin{thm} \label{thm:mainthm}
Let $K \subset \R^d$ be a compact convex set
and $C \subset K$ be a closed subset
such that $\operatorname{bd}(K) \subset C$.
If $f \colon C \to \R$
is continuous and convex
then $f$ extends using the convex roof
construction to a continuous convex
function on $K$.
\end{thm}

Examples show that the theorem is optimal in the sense that
we cannot simply take any other smaller natural subset of the boundary (such
as the extreme points) and always expect the convex roof construction
to be continuous.  Lima~\cite{Lima} proves that a continuous
function defined on a closed subset of the set of extreme points
extends to a continuous convex function of the convex hull.
In particular, every continuous function on the
set of extreme points extends to a continuous convex function
on the convex hull if and only
if the set of extreme points is closed.  In this paper we are more
interested in convex extensions from sets that are larger than the extreme
points.  We are in particular interested in the convex roof
construction and its regularity.

The theorem gives us a necessary and sufficient condition for
a continuous function defined on a closed compact set $C$ to extend
continuously to the convex hull.

\begin{cor}
Let $C \subset \R^d$ be a compact set and let
$f \colon C \to \R$ be a continuous convex function.
Then $f$ extends to a continuous convex function
on the convex hull $K = \operatorname{co}(C)$ if and only
if $f$ extends to a continuous convex function
on $\operatorname{bd}(K)$.
\end{cor}

Example~\ref{noext:example} shows that it is not always possible to extend a
convex function continuously to a convex function of the boundary.

Let us give a rather interesting corollary of the main theorem.
We will also give a simple independent proof of this result.  A
set $K$ is \emph{strictly convex} if the relative boundary
$\operatorname{bd}(K)$
does not contain any intervals.  Notice
that any function defined on the boundary of a strictly convex set
is automatically convex.

\begin{cor}
Let $K \subset \R^d$ be a compact strictly convex set.
If $f \colon \operatorname{bd}(K) \to \R$
is continuous, then $f$ extends using the convex roof construction
to a continuous convex function on $K$.
\end{cor}

It is natural to ask about regularity.  Example~\ref{noreg:example}
shows a $C^\infty$ function $f$ on a strictly convex
$\operatorname{bd}(K)$ such that no convex extension of
$f$ is Lipschitz on $K$; the derivative
must blow up when approaching the boundary.

However, the derivative blowing up at a boundary point is the worst that can
happen.  We say that a convex function $f \colon \operatorname{bd}(K) \to \R$ has a
\emph{nonvertical supporting hyperplane} at $p \in \operatorname{bd}(K)$ if there exists an
affine function $A \colon \R^d \to \R$ such that $A(x) \leq f(x)$ for all
$x \in \operatorname{bd}(K)$
and such that $A(p) = f(p)$.
This is eqivalent to $A(x) = L(x)+f(p)$, where $L$ is a subgradient of $f$
at $p$, that is, $L$ is an element of $\partial f(p)$.
If the convex roof $\widehat{f}$ is
differentiable at $p \in \operatorname{bd}(K)$, then $f$ has a nonvertical
supporting hyperplane at $p$ (the tangent hyperplane to the graph of $\widehat{f}$ at $p$).
It turns
out this is also a sufficient condition for $\widehat{f}$ to be continuously
differentiable on $K$.

\begin{thm} \label{thm:reg}
Let $K \subset \R^d$ be a compact convex set, $E \subset
\operatorname{bd}(K)$
a closed subset
such that $\operatorname{bd}(K)$ is
a $C^1$ manifold near every point of $\operatorname{bd}(K) \setminus E$.
Suppose that $f \colon \operatorname{bd}(K) \to \R$
is a convex function bounded from below
that is $C^1$ on $\operatorname{bd}(K) \setminus E$.  Then
\begin{enumerate}[(i)]
\item $\widehat{f} \colon K \to \R$ is $C^1$ on
$\operatorname{ri}(K) \setminus \operatorname{co}(E)$.
\item If $f$ has a nonvertical supporting hyperplane at $p \in
\operatorname{bd}(K)
\setminus \operatorname{co}(E)$,
then $\widehat{f}$ is differentiable at $p$.  In fact there exists
a convex neighborhood $U$ of $p$ and
a convex function $g \colon U \to \R$,
${g|}_{U \cap K} = {\widehat{f}|}_{U \cap K}$, and
${g|}_{U \cap K}$ is $C^1$.
\end{enumerate}
\end{thm}

By $\operatorname{bd}(K)$ being a $C^1$ manifold near some $p \in
\operatorname{bd}(K)$ we mean
that there exists a neighbourhood $U$ of $p$ such that $\operatorname{bd}(K) \cap U$
is an embedded $C^1$ submanifold of $U$.

It is possible to construct examples showing that
$\widehat{f}$ need not be differentiable at points
of $\operatorname{co}(E)$.
The construction of the $g$ in the proof together with the
compactness of $K$ yields
the following corollary.

\begin{cor} \label{cor:reg}
Let $K \subset \R^d$ be a compact convex set and suppose that
$\operatorname{bd}(K)$ is a $C^1$ manifold.  Suppose that $f \colon
\operatorname{bd}(K) \to \R$
is a $C^1$ convex function that
has a nonvertical supporting hyperplane at every point of the
boundary. Then
there exists a proper convex function $g \colon \R^d \to \R$
such that ${g|}_K = \widehat{f}$ and
$\widehat{f} \colon K \to \R$ is a $C^1$ function on all of $K$.
\end{cor}

Surprisingly, Example~\ref{example:noc2} shows that this theorem does not
extend to $C^2$ or higher regularity.
We also remark that the $g$ we construct in the proof is the smallest convex
extension to all of $\R^d \setminus K$, and this
$g$ need not be differentiable everywhere.

The authors would like to acknowledge David Meyer for inspiring
suggestions and Jon Grice for
useful discussions.  We are grateful to David Ullrich for
pointing out a subtle typo.


\section{Convex roof extension} \label{section:convexroof}

Let $I$ be nonempty set, not necessarily finite.
Let $\R^{(I)}$ be the real vector space of $I$-tuples of real
numbers indexed by $I$, with only finitely many nonzero components:
\begin{equation}
\R^{(I)}= \{ t = \{t_i\}_{i \in I} \in \R^I : \operatorname{supp}(t) \text{
is finite } \},
\end{equation}
where $\operatorname{supp}(t)$ denotes the set $\{i \in I : 
t_i \neq 0 \}$.  Denote by $\lvert\operatorname{supp}(t)\rvert$ the
cardinality of $\operatorname{supp}(t)$.

Let $\{x_i\}_{i \in I}$ a family of elements 
of the vector space $\R^d$ and
$x \in \operatorname{co}(\{x_i\})$.
Consider the (nonempty by
hypotheses) set $\sC(x)$ of all writings of $x$ as a convex combination of 
$\{ x_i \}$: 
\begin{equation}\label{comb}
\sC(x) := \left\{\{t_i\} \in {[0,1]}^{(I)} :
\sum t_i =1  \text{ and } x = \sum t_i x_i \right\} .
\end{equation}
Of course, $\sC(x)$ depends on $x$, $I$, and $\{ x_i \}_{i\in I}$

\begin{defn}
Let $C \subset \R^d$, and let $f \colon C \to \R$.
Let $I = C$, and $\{x_i\}_{i\in I} = C$ where $x_i = i$.
We define the \emph{convex roof} of $f$ to be the function
$\widehat{f} \colon \operatorname{co}(C) \to \R \cup \{ -\infty \}$ given by
\begin{equation}
\widehat{f}(x) := \inf_{t \in \sC(x)}
\sum_{i \in I} t_i f(x_i) .
\end{equation}
\end{defn}

To make the above definition workable we need the 
classical lemma of Carath\'eodory (see e.g.\ \cite{Rockafellar:book}).
Note that $\sC(x)$ is a convex subset of the real vector space
$\R^{(I)}$.

\begin{lemma}[Carath\'eodory] \label{lemma:cara}
Every element of $\sC(x)$ is a convex combination of elements of $\sC(x)$ with
supports of cardinality at most $d+1$.
\end{lemma}

A sketch of the proof: For every $r \geq 1$, consider the set
\begin{equation}
\sC_{r}(x) :=
\left\{\{t_i\} \in {[0,1]}^{(I)}
: \lvert\operatorname{supp}(t_i)\rvert \leq r,
\sum t_i = 1  \text{ and } x = \sum t_i x_i \right\} .
\end{equation}
One shows
using the standard method (see e.g.\ \cite{Rockafellar:book}) that for every
$r \geq d+1$, any element in $\sC_{r+1}(x)$ is an average of two elements in
$\sC_r(x)$.

In particular, the set $\sC_{d+1}(x)$ is nonempty.
An easy computation obtains the following standard corollary.

\begin{cor}
Let $C \subset \R^d$, and $f \colon C \to \R$ be bounded.
As before, let $I = C$, and $\{ x_i \}_{i\in I} = C$ where
$x_i = i$.  Then
\begin{equation}
\widehat{f}(x) = \inf_{t \in \sC_{d+1}(x)}
\sum_{i \in I} t_i f(x_i) .
\end{equation}
\end{cor}

The following facts are standard and not difficult to prove.

\begin{prop} \label{prop:hatext}
Let $C \subset \R^d$ be a compact set and let
$K = \operatorname{co}(C)$.
If $f \colon C \to \R$
is bounded from below and convex, then
\begin{enumerate}[(i)]
\item $\widehat{f}$ is bounded from below (in particular real valued).
\item if $f$ is also bounded from above, then $\widehat{f}$ is bounded.
\item $f = {\widehat{f}|}_{C}$.
\item $\widehat{f}$ is convex.
\item If $g \colon K \to \R$ is convex and ${g|}_C \leq f$,
then $g \leq \widehat{f}$.
\end{enumerate}
\end{prop}

\begin{lemma} \label{lemma:hatcontext}
Let $C \subset \R^d$ be a compact set, and $K = \operatorname{co}(C)$.
If $f \colon C \to \R$
is bounded, lower semicontinuous, and convex, then
$\widehat{f}$ is lower semicontinuous.
\end{lemma}

The lemma follows easily from the results in \cite{Rockafellar:book},
however, the idea in the following proof will be useful and so we include
it.

\begin{proof}
Take a sequence $\{ x_i \}$ in $K$ that converges to $x \in K$.
Using Carath\'eodory's
lemma, for each  $x_i$,
find $d+1$ points $x_i^1,\ldots,x_i^{d+1} \in C$ and
$t_i^1,\dots,t_i^{d+1} \in [0,1]$ such that
$\sum_j t_i^j = 1$ and such that
$x_i = \sum_j t_i^j x_i^j$.
By compactness of $K$ and the corollary to
the Carath\'eodory's lemma we assume that
\begin{equation}
\widehat{f}(x_i) = \sum_{j=1}^{d+1} t_i^j f(x_i^j) .
\end{equation}
For any subsequence of $\{x_i\}$ we can take a further subsequence where 
$x_i^j$ and $t_i^j$ converge to $x^j$ and $t^j$ respectively
(as $K$ is compact).
\begin{equation}
\liminf_{i \to \infty} \widehat{f}(x_i) =
\liminf_{i \to \infty} 
\sum_{j=1}^{d+1} t_i^j f(x_i^j)
\geq
\sum_{j=1}^{d+1} t^j f(x^j) 
\geq
\widehat{f}(x).
\end{equation}
\end{proof}


\section{Continuity of the extension} \label{section:controof}

A point $p \in K$ ($K$ convex) is said to be \emph{extreme}
if $p = tx+(1-t)y$, $t\in[0,1]$, $x,y \in K$ implies
that either $p=x$ or $p=y$.
We always get continuity of the extension at the extreme points,
even without requiring that $f$ be defined on the entire relative
boundary.

\begin{lemma} \label{lemma:contatextreme}
Let $C \subset \R^d$ be a compact convex set and
let $K = \operatorname{co}(C)$.
If $f \colon C \to \R$ is continuous and convex and $p$
is an extreme point of $K$, then
$\widehat{f}$ is continuous at $p$.
\end{lemma}

\begin{proof}
As in the proof of Lemma~\ref{lemma:hatcontext} we take
a sequence $\{ x_i \}$ in $K$ that converges to $p$.
Using Carath\'eodory's lemma and its corollary, the compactness of
$K$, and taking subsequence of a subsequence we assume that
$x_i = \sum_{j=1}^{d+1} t_i^j x_i^j$, where
\begin{equation}
\widehat{f}(x_i) = \sum_{j=1}^{d+1} t_i^j f(x_i^j) .
\end{equation}
Also assume that $x_i^j$ converges to $x^j$ and $t_i^j$ converges to $t^j$.
Thus,
\begin{equation}
\lim_{i\to\infty}
\widehat{f}(x_i) =
\lim_{i\to\infty}
\sum_{j=1}^{d+1} t_i^j f(x_i^j) 
=
\sum_{j=1}^{d+1} t^j f(x^j) .
\end{equation}
We also obtain that $p = \sum t^j x^j$.
As $p$ is an extreme point this means that $x^j = p$ (or $t^j = 0$) for all $j$, therefore
$\lim \widehat{f}(x_i) = \sum t^j f(x^j) = f(p) = \widehat{f}(p)$.
\end{proof}

To obtain continuity of the extension at nonextreme points
we require the following lemma.

\begin{lemma} \label{lemma:bndrycont}
Let $K \subset \R^d$ be a compact convex set and let
$f \colon K \to \R$ be a convex function
lower semicontinuous at $p \in \operatorname{bd}(K)$,
such that
${f|}_{\operatorname{bd}(K)}$ is continuous at $p$.
Then $f$ is continuous at $p$.
\end{lemma}

\begin{proof}
Suppose for contradiction
that $f$ is not continuous at $p \in \operatorname{bd}(K)$.
As $f$ is lower semicontinuous at $p$,
there must exist a $\delta > 0$ and a sequence $x_j \in K$ converging to $p$
such that
$f(x_j) \geq f(p) + \delta$ for all $j$.
We pick a fixed point $y \in \operatorname{ri}(K)$.
Let $y_j \in \operatorname{bd}(K)$ be the unique point in
$\operatorname{bd}(K)$ on the line through $y$ and $x_j$ such that $x_j$ lies on
the line segment $[y_j,y]$.  It is clear that $\lim y_j = p$.
As ${f|}_{\operatorname{bd}(K)}$ is continuous at $p$,
then for
large enough $j$ we have
$f(y_j) \leq f(p) + \delta/2$.

As $f$ is convex, then for $t \geq 1$ we have
\begin{equation}
(1-t)f(y_j) + t f(x_j) \leq
f\bigl((1-t)y_j + t x_j\bigr),
\end{equation}
whenever $(1-t)y_j + t x_j \in K$.
Let $t_j > 1$ be such that
$y = (1-t_j)y_j + t_j x_j$.
We have $t_j \to \infty$
as $\lim x_j = \lim y_j = p$.
For all $j$ we have
\begin{equation}
(1-t_j)f(y_j) + t_j f(x_j) \leq f(y) .
\end{equation}

For large enough $j$ we have
$f(y_j) \leq f(p) + \delta/2$ and
$f(x_j) \geq f(p) + \delta$.  Hence
\begin{equation}
f(p) + \delta/2
+ t_j \delta/2
=
(1-t_j)\bigl(f(p) + \delta/2\bigr)
+ t_j \bigl(f(p)+\delta\bigr)
\leq f(y) .
\end{equation}
Now $t_j \to \infty$ obtains a contradiction.
\end{proof}

Although we will not need it,
let us remark that the lemma is really local.  That is, compactness of $K$ is
not necessary.  We can apply it to any closed convex set $K$ by taking
a closed ball $B$ centered at $p$ and applying the theorem to
$B \cap K$.

We now have all the tools to prove the main theorem.

\begin{proof}[Proof of Theorem~\ref{thm:mainthm}]
We know that $\widehat{f}(x) = f(x)$ for $x \in C$
by Proposition~\ref{prop:hatext} and
by Lemma~\ref{lemma:hatcontext}, $\widehat{f}$ is
lower semicontinuous.
As $\widehat{f}$ is convex, it is standard that $\widehat{f}$ is continuous
on $\operatorname{ri}(K)$.
Then we apply
Lemma~\ref{lemma:bndrycont}.
\end{proof}


\section{Regularity of the convex roof} \label{section:regularity}

Let $f \colon K \to \R$ be a convex function.
Define a \emph{subgradient} at $p \in K$ to be a linear mapping $L$
such that $L(x-p) + f(p) \leq f(x)$ for all $x \in K$.
We need some classical results about derivatives of convex
functions.

\begin{thm}[See e.g.\ Theorems 25.1 and 25.5 in \cite{Rockafellar:book}]
\label{subgrad:thm}
Let $U \subset \R^d$ be an open convex set and let $f \colon U \to \R$ be a convex
function.
\begin{enumerate}[(i)]
\item
$f$ has a unique subgradient at $p$ if and only if
$f$ is differentiable at $p$.
\item
$f$ is differentiable on a dense subset $D \subset U$.
\item
The mapping $x \mapsto \nabla f(x)$ is continuous on $D$.
\end{enumerate}
\end{thm}

Before we prove
Theorem~\ref{thm:reg}, let us note the following observation (see
also~\cite{uhlmann2}, although Uhlmann only considers continuous $f$).

\begin{prop} \label{prop:el}
Let $K \subset \R^d$ be a compact convex set, $f \colon \operatorname{bd}(K) \to
\R$ a convex function bounded from below, and $p \in \operatorname{ri}(K)$.
There exists a closed convex set $W \subset K$,
with $p \in W$ and $W = \operatorname{co}\bigl(W \cap
\operatorname{bd}(K)\bigr)$ with the
following property.  If $A$ is any affine function with $A(p) =
\widehat{f}(p)$ and $A(x) \leq \widehat{f}(x)$ for all $x \in K$, then
\begin{equation}
W \subset \{ x \in K : A(x)=\widehat{f}(x) \} .
\end{equation}
\end{prop}

In particular, the proposition says that there exists a line $\ell \subset
\R^d$ through $p$ such that for any subgradient $L$ of
$\widehat{f}$ at $p$, we get that
\begin{equation}
L(x-p)+\widehat{f}(p) = \widehat{f}(x)
\end{equation}
for all $x \in \ell \cap K$.

\begin{proof}
Let $p \in \operatorname{ri}(K)$.  By definition of convex roof and the
lemma of Carath\'eodory we can find sequences
$t_{i}^1, t_i^2,\ldots, t_i^{d+1} \in [0,1]$, and
$x_{i}^1, x_i^2,\ldots, x_i^{d+1} \in \operatorname{bd}(K)$ with $\sum_j t_i^j = 1$
and $p = \sum_j t_i^j x_i^j$, and such that
\begin{equation}
\widehat{f}(p) = \lim_{i\to \infty} \sum_{j=1}^{d+1} t_i^j f(x_i^j) .
\end{equation}
By compactness of $K$ we can assume that there exist $t^1,\ldots,t^{d+1} \in
[0,1]$
and $x^1,\ldots,x^{d+1} \in \operatorname{bd}(K)$ such that
$t_i^j$ and $x_i^j$ converge to $t^j$ and $x^j$ respectively as $i$ goes
to infinity.
Furthermore, the sequence $\{ f(x_i^j) \}_{i=1}^\infty$
must be bounded above and as $f$
is also bounded from below and $K$ is compact, we
can assume that the limit of $f(x_i^j)$ exists.

Write
\begin{equation}
w_j = \lim_{i\to\infty} f(x_i^j) .
\end{equation}
We see that 
\begin{equation}
\widehat{f}(p) = \sum_{j=1}^{d+1} t^j w_j .
\end{equation}
Let $C = \operatorname{co}(\{x^1,x^2,\ldots,x^{d+1}\})$.
If $A$ is an affine function such that $A(x) \leq \widehat{f}(x)$ and
$A(p) = \widehat{f}(p)$, then $A(x^j) \leq w_j$.  We can now without loss
of generality assume that $t^j \not= 0$, and therefore $p \in
\operatorname{ri}(C)$.  As $A(p) = \widehat{f}(p)$ we see that 
if $x = \sum_j s^j x^j$ (with $x \in \operatorname{ri}(C)$)
where $\sum s^j = 1$ then
\begin{equation}
A(x) = \sum_{j=1}^{d+1} s^j w^j .
\end{equation}
For any $\epsilon > 0$ then for large enough $i$ we have
\begin{equation}
\widehat{f}\Bigl(\sum_j s^j x_i^j\Bigr) \leq \sum_{j=1}^{d+1} s^j f(x_i^j) \leq
\widehat{f}\Bigl(\sum_j s^j x_i^j\Bigr) + \epsilon .
\end{equation}
As $\widehat{f}$ is continuous in $\operatorname{ri}(K)$, then
\begin{equation}
\widehat{f}(x) = A(x)
\end{equation}
for all $x \in \operatorname{ri}(K)$, and we are done.
\end{proof}

We can now prove the first part of Theorem~\ref{thm:reg}.
For this statement we can safely drop the continuity hypothesis for
the derivative.  The $C^1$ hypothesis is necessary to extend
differentiability to the boundary.  On the other hand we automatically
obtain $C^1$ differentiability in the interior by Theorem~\ref{subgrad:thm}.

\begin{lemma} \label{lemma:reg1}
Let $K \subset \R^d$ be a compact convex set, $E \subset
\operatorname{bd}(K)$
a closed set
such that $\operatorname{bd}(K)$ is
a differentiable manifold near every point of $\operatorname{bd}(K) \setminus E$.
Suppose that $f \colon \operatorname{bd}(K) \to \R$
is a convex function bounded from below
that is differentiable on $\operatorname{bd}(K) \setminus E$.
Then $\widehat{f} \colon K \to \R$ is $C^1$ on $\operatorname{ri}(K)
\setminus \operatorname{co}(E)$.
\end{lemma}

\begin{proof}
Let us assume that we work in the affine hull of $K$ and therefore
without loss of generality assume that the
relative interior $\operatorname{ri}(K)$ is in fact equal to the
topological interior.

Let $p \in \operatorname{ri}(K) \setminus \operatorname{co}(E)$.
Suppose that $L$ and $\widetilde{L}$
are subgradients of $\widehat{f}$ at $p$.
By Proposition~\ref{prop:el} there is a line
$\ell \subset \R^d$ (with $p \in \ell$) such that
\begin{equation}
\widehat{f}(x) = L(x-p)+\widehat{f}(p) = \widetilde{L}(x-p)+\widehat{f}(p)
\end{equation}
for all $x \in \ell$.
We can also assume 
that there exists a
$q \in \ell \cap (\operatorname{bd}(K) \setminus E)$ as $p$ was not in the convex
hull of $E$.
We notice that $\ell$ is not tangent to $\operatorname{bd}(K)$ at $q$
as $K$ is convex; if $\ell$ were tangent then $\ell \cap K \subset
\operatorname{bd}(K)$, which contradicts the fact that $p \in \operatorname{ri}(K)$.

Let $T$ denote the $(d-1)$-dimensional affine manifold in $\R^d \times \R$
through the point $(p,f(p))$,
tangent to the graph of $f$
\begin{equation}
\Gamma_f = \{ (x,f(x)) \in \R^d \times \R : x \in \operatorname{bd}(K) \}.
\end{equation}
Note that near $(p,f(p))$, $\Gamma_f$ is a $C^1$ manifold.
And define the line $\tilde{\ell} \in \R^d \times \R$ as
\begin{equation}
\widehat{\ell} = \{ (x,L(x-p)+\widehat{f}(p)) = (x,\widehat{f}(x)) \in \R^d \times \R : x
\in \ell \} .
\end{equation}
Let $W$ be the affine span of $T$ and $\widehat{\ell}$ (that is, the smallest affine
space containing both $T$ and $\widehat{\ell}$).  As $\ell$ was
not tangent to $\operatorname{bd}(K)$, then $\widehat{\ell}$ is not contained in $T$ and
therefore $W$ is $d$-dimensional.

Let $\Gamma_L$ be the graph of the mapping
$x \mapsto L(x-p) +\widehat{f}(p)$, that is
\begin{equation}
\Gamma_L = \{ (x,y) \in \R^d \times \R : y = L(x-p)+\widehat{f}(p) \} .
\end{equation}
Similarly let $\Gamma_{\widetilde{L}}$ be the graph of the mapping
$x \mapsto \widetilde{L}(x-p) +\widehat{f}(p))$.
As for $x \in
\operatorname{bd}(K)$ we have
$\widetilde{L}(x-p) + \widehat{f}(p) \leq f(x)$ and
$L(x-p) + \widehat{f}(p) \leq f(x)$ we see
that $T$ is a subset of
$\Gamma_L \cap \Gamma_{\widetilde{L}}$.
Furthermore, $\widehat{\ell} \subset \Gamma_L \cap \Gamma_{\widetilde{L}}$,
by definition of $\ell$.
Therefore, $W \subset \Gamma_L \cap \Gamma_{\widetilde{L}}$.

As $\Gamma_L$, $\Gamma_{\widetilde{L}}$, and $W$ are all $d$-dimensional
affine subspaces, we see that
\begin{equation}
\Gamma_L = \Gamma_{\widetilde{L}} = W .
\end{equation}
So $L = \widetilde{L}$ and 
$\widehat{f}$ is differentiable at $p$
by Theorem~\ref{subgrad:thm}.
\end{proof}

Let us prove the second part of the theorem.

\begin{lemma} \label{lemma:reg2}
Let $K \subset \R^d$ be a compact convex set, $E \subset
\operatorname{bd}(K)$
a closed set
such that $\operatorname{bd}(K)$ is
a $C^1$ manifold near every point of $\operatorname{bd}(K) \setminus E$.
Suppose that $f \colon \operatorname{bd}(K) \to \R$
is a convex function bounded from below that is $C^1$ on
$\operatorname{bd}(K) \setminus E$
and such that
$f$ has a nonvertical supporting hyperplane at $p \in \operatorname{bd}(K)
\setminus \operatorname{co}(E)$.

Then $\widehat{f}$ is differentiable at $p$.  In fact there exists
a convex neighborhood $U$ of $p$ and
a convex function $g \colon U \to \R$, such that
${g|}_{U \cap K} = {\widehat{f}|}_{U \cap K}$, and
${g|}_{U \cap K}$ is $C^1$.
\end{lemma}

\begin{proof}
Again without loss of generality assume that $\operatorname{ri}(K)$
is an open set in $\R^d$.  We pick a point $p \in \operatorname{bd}(K) \setminus
\operatorname{co}(E)$.  Near $p$, the graph of $f$ above
$\operatorname{bd}(K)$ is a $C^1$ manifold.  As $f$ has a
nonvertical supporting hyperplane at $p$, then it has such a hyperplane
at all points near $p$; for example, we can simply take tangent hyperplanes
to some $C^1$ manifold of dimension $d$ in $\R^d \times \R$ that contains
the graph of $f$ and is tangent to the supporting hyperplane at $p$.

Take a small convex neighborhood $U$ of $p$, such that
$U \cap \operatorname{bd}(K)$ is a connected $C^1$ manifold and such
that $f$ has a nonvertical supporting hyperplane at each
$q \in U \cap \operatorname{bd}(K)$.
At each $q \in U \cap \operatorname{bd}(K)$ we take the unique
supporting hyperplane given by an affine function
$L_q \colon \R^d \to \R$ such that
$L_q$ minimizes the derivative in the normal direction to
$\operatorname{bd}(K)$.
The magnitude of the gradient of $L_q$ must be 
uniformly bounded for $q \in U \cap \operatorname{bd}(K)$ by the same argument
that guaranteed nonvertical supporting hyperplanes.
Define
\begin{equation}
g(x) :=
\begin{cases}
\sup \{ L_q(x) : q \in U \cap \operatorname{bd}(K) \} & \text{if $x \in U \setminus K$}, \\
\widehat{f}(x) & \text{if $x \in U \cap K$}.
\end{cases}
\end{equation}
It is not hard to check that $g$ must be a convex function.
Furthermore, $g$ has a unique subgradient at $p$ (and in fact
at all $q \in U \cap \operatorname{bd}(K)$).  To see this fact, note that
any other possible subgradient must be tangent to the graph
of $f$ over $\operatorname{bd}(K)$ and hence would give a candidate for $L_p$,
which is unique.
\end{proof}

We now prove Corollary~\ref{cor:reg}.  We look at how the
function $g$ was constructed above.  We assume that
the boundary $\operatorname{bd}(K)$ is $C^1$ and
the function $f$ is $C^1$ on all of $\operatorname{bd}(K)$.
As $\operatorname{bd}(K)$ is compact,
the magnitude of the gradient
of $L_q$ must be uniformly bounded for all $q \in \operatorname{bd}(K)$.
As $g$ was constructed by taking a supremum over $L_q$,
this supremum must be bounded on all of $\R^d$.

The function $g$ constructed above is the smallest possible
extension to $\R^d \setminus K$, because every convex extension
must lie above the subgradients.  To see that $g$ need not
be differentiable take the function whose graph is a union of lines
all going through a single point $p \notin K$.  We can arrange such
a function to be convex and $C^1$ on $K$.  It is then
obvious that the function is equal to the $g$ above and we can
arrange it to not be differentiable at $p$.


\section{Examples} \label{section:examples}

\begin{example}[Punctured tomato can] \label{tomato:example}
Defining $f$ on only a subset of the boundary (including the extremal set)
is not enough to guarantee continuity of the convex roof.
Let $(x,y,z) \in \R^3$ be our coordinates.
Let
\begin{equation}
\begin{aligned}
& K_1 = \{ (x,y,z) \mid x=-1, y^2+z^2 = 1 \}, \\
& K_2 = \{ (x,y,z) \mid x=1, y^2+z^2 = 1 \}, \\
& K_3 = \{ (0,0,1) \} .
\end{aligned}
\end{equation}
The convex hull of the union is the cylinder
\begin{equation}
\operatorname{co}(K_1 \cup K_2 \cup K_3)
=
\{ (x,y,z) \mid -1 \leq x \leq 1, y^2+z^2 \leq 1 \}.
\end{equation}

Define $f$ to be identically 1 on $K_1$ and $K_2$
and let $f$ be zero on $K_3$.  Then $f$ is continuous
and convex on $K_1 \cup K_2 \cup K_3$.
Lemma~\ref{lemma:hatcontext} tells us that $\widehat{f}$ is
lower semicontinuous.

The convex roof $\widehat{f}$ must be identically 1
on the lines $\{ (x,y,z) \mid -1 \leq x \leq 1, y= y_0, z=z_0 \}$
where $y_0^2+z_0^2 = 1$ are fixed and $y_0 \not= 0$.
On the other hand, at the point $(0,0,1)$, which lies
on the line
$\{ (x,y,z) \mid -1 \leq x \leq 1, y= 0, z=1 \}$, the
function $\widehat{f}$ must be 0.  Hence $\widehat{f}$ cannot
be upper semicontinuous at $(0,0,1)$.

In this case,
note that a continuous extension does exist.  For example
the function $x^2$.  We have only shown that the convex roof construction is
not continuous.
\end{example}

\begin{example} \label{nonclosed:example}
As the set of extreme points may not be closed, when we define a function
only on the set of extreme points, it is natural to require that the
function be uniformly continuous such that it extends to the closure of
the extreme points.  In this case however the extension to
the closure can fail to be convex.  For example take again
$(x,y,z) \in \R^3$ be our coordinates and
let
\begin{equation}
\begin{aligned}
& K_1 = \{ (-1,0,0) \}, \\
& K_2 = \{ (x,y,z) \mid x=0, (y-1)^2+z^2 = 1 , y \not= 0 \}, \\
& K_3 = \{ (1,0,0) \} .
\end{aligned}
\end{equation}
Define the function to be 0 at $K_1$ and $K_3$, and let it be identically 1
on $K_2$.  The function is convex on $K_1 \cup K_2 \cup K_3$ and uniformly
continuous, however the continuous extension to
$K_1 \cup K_2 \cup K_3$ fails to be convex because the function
will be 1 at $(1,0,0)$, while it will be 0 at $(-1,0,0)$ and $(1,0,0)$.
\end{example}

A natural question to ask is what happens then when the function is
defined, continuous, and convex on the closure of the extreme points.
Lima~\cite{Lima} shows that
then any continuous function on the set of extreme points extends
to a continuous convex function on the convex hull
if and only if the set of extreme points is closed.
Let us therefore see an example where the set of extreme points is not
closed.
Let us combine examples \ref{nonclosed:example} and \ref{tomato:example}.

\begin{example}
Take again
$(x,y,z,w) \in \R^4$ be our coordinates and
let
\begin{equation}
\begin{aligned}
& K_1 = \{ (x,y,z,w) \mid x=-1, y^2+(z-1)^2 = 1 , w = 0 \}, \\
& K_2 = \{ (x,y,z,w) \mid x=0,  y^2+(w-1)^2 = 1 , z = 0 \}, \\
& K_3 = \{ (x,y,z,w) \mid x=1,  y^2+(z-1)^2 = 1 , w = 0 \} .
\end{aligned}
\end{equation}
Define the function $f$ to be identically 1 on $K_1$ and $K_3$,
and let it be identically 0 on $K_2$.  Let $K =
\operatorname{co}(K_1 \cup K_2 \cup K_3)$.  The set of
extreme points of $K$ is $K_1 \cup K_2 \cup K_3 \setminus \{ (0,0,0,0) \}$.
The function $f$ is continuous, convex, and defined precisely on
the closure of the extreme points.

As $\{ w=0 \}$ is a supporting hyperplane,
the convex roof
construction done in $\{ w = 0 \}$ is equal to the convex roof construction
done in $\R^4$ and restricted to 
$\{ w = 0 \}$.  On $\{ w = 0 \}$ we are 
in the situation of Example~\ref{tomato:example}, and thus
$\widehat{f}$ is not continuous.
\end{example}

In the above two examples, there always existed some continuous convex
extension.  However, the following modification of
Example~\ref{tomato:example} shows that this is not always true either.

\begin{example} \label{noext:example}
Let $(x,y,z) \in \R^3$ be our coordinates.
Let
\begin{equation}
\begin{aligned}
& K_1 = \{ (x,y,z) \mid -1 \leq x \leq -z , y^2+z^2 = 1 \} \cup
 \{ (x,y,z) \mid x=1, y^2+z^2 = 1 \}, \\
& K_2 = \{ (0,0,1) \} .
\end{aligned}
\end{equation}
See Figure~\ref{fig:tomatocan}.
The convex hull of the union is again cylinder
\begin{equation}
\operatorname{co}(K_1 \cup K_2)
=
\{ (x,y,z) \mid -1 \leq x \leq 1, y^2+z^2 \leq 1 \}.
\end{equation}

\begin{figure}[ht]
\includegraphics{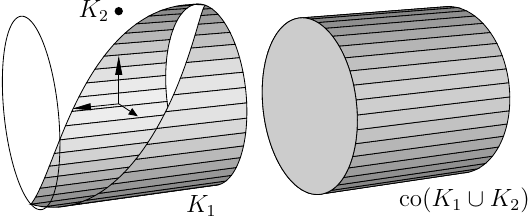}
\caption{The sets $K_1$ and $K_2$ and their convex hull.\label{fig:tomatocan}}
\end{figure}

Define $f$ to be identically 1 on $K_1$,
and let $f$ be zero on $K_2$.  Then $f$ is obviously continuous.
Also $f$ is convex as it is the restriction of the convex
roof construction from Example~\ref{tomato:example}.

On a line 
$\{ (x,y,z) \mid -1 \leq x \leq 1, y= y_0, z=z_0 \}$, $y_0 \not= 0$,
any convex extension must be identically 1,
since $f$ is 1 at the endpoints and
identically 1 for all $x \in [-1,-z_0]$.  Therefore, any
convex extension will fail to be continuous at $(0,0,1)$.
\end{example}

\begin{example} \label{noreg:example}
Let us construct a
$C^\infty$ function on the set where $x^4+y^4=1$ whose convex roof extension
is not Lipschitz on the convex hull.

Let $K = \{ (x,y) : x^4+y^4 \leq 1 \}$.
Define $f \colon \operatorname{bd}(K) \to \R$ by 
$f(x,y) = 1-\sqrt{1-y^4}$.
The function $f$ is $C^\infty$; the only issue arises at $(0,\pm 1)$,
by parametrizing $\operatorname{bd}(K)$ near $(0,1)$ by $(x,{(1-x^4)}^{1/4})$,
we notice that on $\operatorname{bd}(K)$ near $(0,1)$, $f$
is given by $1-x^2$.

It is also not hard to see that
$\widehat{f}(x,y) = 1-\sqrt{1-y^4}$.  This is because for a
fixed $y$, being constant is the largest that $\widehat{f}$ can be and
$\widehat{f}$ is the largest convex extension.  But that means that the
derivative of $\widehat{f}$ goes to infinity when we approach $(0,\pm 1)$
from
the inside of the disc, and hence $\widehat{f}$ is not Lipschitz on
$K$.  See Figure~\ref{fig:potatochip}.

\begin{figure}[ht]
\includegraphics{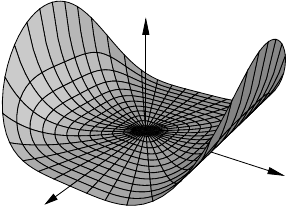}
\caption{$\widehat{f}(x,y) = 1-\sqrt{1-y^4}$ on $x^4+y^4\leq 1$.\label{fig:potatochip}}
\end{figure}

As $\widehat{f}$ is the largest convex extension,
it is not hard to see that any other convex extension of $f$ is also
not Lipschitz (the derivative must also blow up at
$(0,\pm 1)$).
\end{example}

\begin{example} \label{example:noc2}
It is rather surprising that even if $f$ is $C^\infty$ on
the boundary, the convex roof need not be $C^2$ on the interior
(though it must be $C^1$).
Let $K = \{ (x,y) : x^2+y^2 \leq 1 \}$.
Let $f \colon \operatorname{bd}(K) \to \R$ be defined
by $(x+1)x^2$.  Obviously $f$ is $C^\infty$ (in fact real-analytic)
on $\operatorname{bd}(K)$.  By Theorem~\ref{thm:reg} $\widehat{f} \in C^1(K)$.

We now note that on the convex hull of the points $(0,1)$, $(0,-1)$, and
$(-1,0)$, the convex roof $\widehat{f}$ must be identically zero.  On the
other hand, by a similar argument as above, for $x \geq 0$, we have
$\widehat{f}(x,y) = (x+1)x^2$.  This means that for example on the line
$y=0$, the function $\widehat{f}$ is identically zero for $x \leq 0$ and
$(x+1)x^2$ for $x \geq 0$.  Therefore $\widehat{f}$ is not $C^2$ at the
origin (and in fact along the line $x=0$).
\end{example}


\section{Application to quantum computing} \label{section:apptoquant}

The convex roof extension was  introduced by Uhlmann~\cite{uhlmann2} as a
way of defining a different measure of entanglement on the set of density
operators.

Given a quantum system there is associated to it a
complex Hilbert space $\sH$
so that all the possible states of that quantum system
are in 1-1 correspondence with the set of 1-dimensional subspaces of $\sH$,
that is, the projective space
$\bP(\sH)$.  By assigning the orthogonal projection onto
a 1-dimensional subspace we obtain a bijection between such orthogonal
projectors
and $\bP(\sH)$.   In particular if $\rho$ is such a projector
it has trace 1.
 
Mixed states of a quantum system are probabilistic averages of pure states.
The mixed states of a quantum system are in 1-1
correspondence with the set of positive semidefinite operators on
$\sH$ of trace $1$.  This set is a compact convex subset
with nonempty interior of the set of Hermitian operators of trace $1$,
which in turn is an affine subspace of the set of complex linear
operators on $\sH$.

Any positive semidefinite operator of trace 1 is a convex
combination of projectors of rank 1.
Hence the set extreme points of the convex set of states (the mixed
states) consists of the orthogonal projectors of rank $1$, the so-called
pure states.
Note also that for any positive
semidefinite operator $\rho$ of trace 1 we have the inequality
$\operatorname{Trace}(\rho^2) \leq 1$ with equality if and only if $\rho$ is
a projector of rank 1.  Therefore the pure states are the relative boundary
of the set of all states.
The set of pure states can be identified with the set of
nonzero vectors in $\sH$ modulo multiplication by nonzero complex
scalars.

Suppose that $\sH$ is a tensor product of two
Hilbert spaces $\sH_1$, $\sH_2$, 
\begin{equation}
\mathcal{H} = \mathcal{H}_1\otimes \mathcal{H}_2 .
\end{equation}
A pure state given by the vector $v$ is called a product state (or
unentangled) if $v$ is a tensor product of two vectors
$v_i\in \mathcal{H}_i$ 
\begin{equation}
v = v_1 \otimes v_2 .
\end{equation}
This decomposition
does not depend on the representative chosen; a
vector is a tensor product then any multiple of it is again a tensor
product. 

We call a pure state entangled if (any) representative vector is
\emph{not} a tensor product. 
We consider functions defined on the set of pure states that are zero on the
set of unentangled states and larger than zero on the set of entangled
states.
This is called a measure of entanglement. 

The question posed by Uhlmann is as follows: How to extend a measure of
entanglement from pure states to mixed states?  In other words, how to
extend a function from the (relative) boundary of the set of states to the
whole set.  The method proposed is the convex roof construction. 

This was later applied to extend particular measures of entanglement, for
instance the Von Neumann entropy
(see Nielsen and Chuang~\cite{NielsenChuang}) or linear entropy. 

Let us analyse a bit the common structure for these
measures of entanglement.
Recall (also see~\cite{NielsenChuang}) that for every vector $v$ there
exists an integer $r \geq 1$, strictly positive numbers $\lambda_1 \geq
\lambda_2 \geq \cdots \geq \lambda_r$, and orthonormal
systems $e_1$, \ldots,
$e_r$ in $\mathcal{H}_1$ and $f_1$, \ldots, $f_r$ in $\mathcal{H}_2$ such
that 
\begin{equation}
v = \sum_{i=1}^r \sqrt{\lambda_i} \, e_i \otimes f_i
\qquad \text{(the Schmidt decomposition)}.
\end{equation}
The integer $r$ and the $\lambda_j$ are uniquely determined.
The vector $v$ is a tensor product if and only if $r=1$. 

Assuming $\snorm{v} = 1$
we can measure how ``far'' is $r$ from $1$ by, for example, considering a
function $\phi$ with some extra properties
($\phi(0)=0$, $\phi \geq 0$, etc\ldots)\ and then defining
$f(v) = \sum \phi(\lambda_i)$, or a function of
this sum.  For example, the linear entropy is defined by 
\begin{equation}
f(v) = \sqrt{1- \sum_{i=1}^r \lambda_i^2} .
\end{equation}

The function $f$ is zero if and only if exactly one of the
numbers $\lambda_i$ is $1$. 
One also notices that $f$ is continuous.
In the case when $\sH_1$ and $\sH_2$ are both of dimension $2$,
there exists  a
closed formula for the convex roof extension of the linear entropy due to
Wootters~\cite{Wootters}.
Note that such a closed formula also describes the set of separable mixed
states, that is, mixed states that are convex combination of pure product
states. 

However, in the study of measures of entanglement there is an acute lack of
explicit formulas for the convex roof extensions.
A consequence of our results is that even if explicit formulas are unknown,
still one concludes that all the measures of entanglements defined using the
convex roof construction are continuous if the measure is continuous on
the set of pure states.

\begin{thm}
Let $f$ be a continuous measure of entanglement defined on the pure states,
then the convex roof extension $\hat{f}$ is a continuous function of all
states.
\end{thm}


\def\MR#1{\relax\ifhmode\unskip\spacefactor3000 \space\fi%
  \href{http://www.ams.org/mathscinet-getitem?mr=#1}{MR#1}}

\end{document}